\documentclass[12pt]{amsart}

\newtheorem{lem}{Lemma}[section]
\newtheorem{cor}{Corollary}[section]
\newtheorem{theo}{Theorem}[section]
\newtheorem{rem}{Remark}[section]

\newtheorem{prop}{Proposition}[section]
\newcommand{\ddd}{\raisebox{1.5mm}{\mbox{$\psi \atop \longrightarrow $ }}}

\newcommand{\mf}[1]{\mathfrak{#1}}
\usepackage{amsfonts}
\usepackage{amssymb}
\usepackage{eucal}
\numberwithin{equation}{section}
\title{ON GEODESIC MAPPINGS IN PARTICULAR CLASS OF ROTER SPACES}
\author{Ryszard Deszcz and Marian Hotlo\'{s}}
\begin{document}
\dedicatory{Dedicated to the memory of Professor Witold Roter}
\begin{abstract}
We determine a particular class of Roter type warped product manifolds. We show that every manifold of that class
admits a geodesic mapping onto a some Roter type warped product manifold. Moreover, both geodesically related
manifolds are pseudosymmetric of constant type.
\end{abstract}
\maketitle

\section{Introduction}

Let $(M,g)$ and $(\overline{M},\overline{g})$ be two $n$-dimensional semi-Riemannian manifolds.
A diffeomorphism $h:M\to\overline{M}$ which maps geodesic lines into geodesic lines is called a 
{\it geodesic transformation}, or a {\it geodesic mapping}, or a {\it projective mapping}.

\vspace{2mm}

The well-known result of Beltrami is presented in {\cite[Theorem 10] {HaVerSigma}} as follows:

\vspace{2mm}

\noindent 
Theorem 10 (Beltrami). {\sl{The real space forms constitute the projective class of the locally
Euclidean spaces, or, still, by applying geodesic transformations to locally Euclidean spaces one
obtains spaces of constant curvature and the class of the spaces of constant curvature is closed
under geodesic transformations.}}

\vspace{2mm}

Manifolds satisfying curvature conditions and admitting geodesic transformations were investigated by several authors,
see, e.g., \cite{{DD}, {DD_PMH}, {DH}, {DH1}, {DH2}, {Hint1}, {Hint2}, {M}, {MSV}, {Sinj}, {Venzi_1978}}.
In particular, we have the following extension of the Beltrami's theorem 
{\cite[Theorem 19] {HaVerSigma}}:

\vspace{2mm}

\noindent
Theorem 19 (Sinjukov, Mike\v{s}, Venzi, Defever and Deszcz).
{\sl{If a semi-symmetric Riemannian 
space admits a geodesic transformation onto some other Riemann manifold, then 
this latter manifold must itself be pseudo-symmetric, and, if a pseudo-symmetric Riemannian space
admits a geodesic transformation onto some other Riemannian manifold, then this latter 
manifold must itself also be pseudo-symmetric}}.

\vspace{2mm}

Thus we can state that the class of pseudosymmetric manifolds is the widest known class of manifolds which
is closed with respect to geodesic mappings. 
It is known that
the curvature tensor of certain non-conformally flat and non-quasi-Einstein pseudosymmetric manifolds of dimension $\geqslant 4$,
is a linear combination of some Kulkarni-Nomizu tensors formed by the Ricci tensor and the metric tensor of the considered manifolds.
A semi-Riemannian manifold with the curvature tensor having this property is named the Roter type manifold.
Evidently, every

\hspace*{0pt}\hrulefill\hspace*{150mm}\\
{\footnotesize{\indent
{\bf{Mathematics Subject Classification (2010).}} 
Primary: 53B20. Secondary: 53C21.

{\bf{Key words and phrases:}} 
geodesic mapping, warped product manifold, Einstein manifold, quasi-Einstein manifold,
Roter type manifold, pseudosymmetric manifold, pseudosymmetry type curvature condition.}}

\newpage

\noindent
Roter type manifold is pseudosymmetric. The converse statement is not true.
It seems that the Roter type manifolds form an important and interesting class of manifolds for study.
In particular, we can consider the following problems related to geodesic mappings of these mani\-folds.

(i) Does admit a Roter type manifold a geodesic mapping?

(ii) If a Roter type manifold $(M,g)$ admits a geodesic mapping onto some manifold $(\overline{M},\overline{g})$,
then in view of the above mentioned theorem $\overline{M}$ is pseudosymmetric. 
Therefore, it is natural to ask as follows: is $\overline{M}$ also
a Roter type manifold?

In this paper we answer to these questions.
First of all, we construct warped product manifolds, with $2$-dimensional base and with fiber of constant curvature, which are Roter
type manifolds and admit geodesic mappings. Moreover, we prove that manifolds geodesically related to these warped products are
also Roter type manifolds. Furthermore, we derive some curvature conditions of pseudosymmetry type which are satisfied by
constructed manifolds.

Continuing the study on geodesic mappings in Roter spaces we obtained also some new results.

\section{Preliminary results}

Let $(M,g)$, $n = \dim M \geqslant 3$, be a semi-Riemannian manifold.
We denote by $\nabla$, $R$, $S$, $\kappa$ and $C$ 
the Levi-Civita connection, the Riemann-Christoffel curvature tensor, the Ricci tensor,
the scalar curvature and the Weyl conformal curvature tensor of $(M,g)$, respectively.
Throughout this paper all manifolds are assumed
to be connected paracompact manifolds of class $C^{\infty }$.\\
Let $\Xi (M)$ be the Lie algebra of vector fields on $M$.
We define on $M$ the endomorphisms
$X \wedge _{A} Y$ and
${\mathcal{R}}(X,Y)$
of $\Xi (M)$ by 
$(X \wedge _{A} Y)Z = A(Y,Z)X - A(X,Z)Y$ and
\begin{eqnarray*}
{\mathcal R}(X,Y)Z
&=&
\nabla _X \nabla _Y Z - \nabla _Y \nabla _X Z - \nabla _{[X,Y]}Z ,
\end{eqnarray*}
respectively, where $A$ is a symmetric $(0,2)$-tensor on $M$
and $X, Y, Z \in \Xi (M) $.
The Ricci tensor $S$,
the Ricci operator ${\mathcal{S}}$,
the tensor $S^{2}$ and the scalar curvature $\kappa $
of $(M,g)$ are defined by
$S(X,Y) = \mathrm{tr} \{ Z \rightarrow {\mathcal{R}}(Z,X)Y \}$,
$g({\mathcal{S}}X,Y) = S(X,Y)$,
$S^{2}(X,Y) = S({\mathcal{S}}X,Y)$ and
$\kappa \, =\, \mathrm{tr}\, {\mathcal{S}}$, respectively.
The endomorphism ${\mathcal{C}}(X,Y)$ of $(M,g)$, $n \geq 3$, is defined by
\begin{eqnarray*}
{\mathcal{C}}(X,Y)Z  &=& {\mathcal{R}}(X,Y)Z
- \frac{1}{n-2} \left( X \wedge _{g} {\mathcal{S}}Y + {\mathcal{S}}X \wedge _{g} Y
- \frac{\kappa}{n-1}X \wedge _{g} Y \right) Z .
\end{eqnarray*}
The $(0,4)$-tensor $G$,
the Riemann-Christoffel curvature tensor $R$ and
the Weyl conformal curvature tensor $C$ of $(M,g)$ are defined by
$G(X_1,X_2,X_3,X_4) = g ( (X_1 \wedge _{g} X_2)X_3,X_4)$,  
\begin{eqnarray*}
R(X_1,X_2,X_3,X_4) \ =\ g({\mathcal{R}}(X_1,X_2)X_3,X_4) ,\ \ \
C(X_1,X_2,X_3,X_4) \ =\ g({\mathcal{C}}(X_1,X_2)X_3,X_4) ,
\end{eqnarray*}
respectively, where $X_1,X_2,X_3,X_4 \in \Xi (M)$.
Let ${\mathcal{B}}$ be a tensor field sending any $X, Y \in \Xi (M)$
to a skew-symmetric endomorphism ${\mathcal{B}}(X,Y)$
and let $B$ be a $(0,4)$-tensor associated with ${\mathcal{B}}$ by
\begin{eqnarray}
B(X_1,X_2,X_3,X_4) &=&
g({\mathcal{B}}(X_1,X_2)X_3,X_4) .
\label{DS5}
\end{eqnarray}
The tensor $B$ is said to be a {\sl{generalized curvature tensor}}
if the following conditions are satisfied
\begin{eqnarray*}
& &
B(X_1,X_2,X_3,X_4) \ =\  B(X_3,X_4,X_1,X_2) ,\\
& &
B(X_1,X_2,X_3,X_4)
+ B(X_3,X_1,X_2,X_4)
+ B(X_2,X_3,X_1,X_4) \ =\  0 .
\end{eqnarray*}
For ${\mathcal{B}}$ as above, let $B$ be again defined by (\ref{DS5}).
We extend the endomorphism
${\mathcal{B}}(X,Y)$ to a derivation
${\mathcal{B}}(X,Y) \cdot \, $
of the algebra of tensor fields on $M$,
assuming that it commutes with contractions and
$\ {\mathcal{B}}(X,Y) \cdot \! f \, =\, 0$, for any smooth function $f$ on $M$.
For a $(0,k)$-tensor field $T$, $k \geqslant 1$,
we can define the $(0,k+2)$-tensor $B \cdot T$ by
\begin{eqnarray*}
& & (B \cdot T)(X_1,\ldots ,X_k,X,Y) \ =\
({\mathcal{B}}(X,Y) \cdot T)(X_1,\ldots ,X_k)\\
&=& - T({\mathcal{B}}(X,Y)X_1,X_2,\ldots ,X_k)
- \cdots - T(X_1,\ldots ,X_{k-1},{\mathcal{B}}(X,Y)X_k) .
\end{eqnarray*}
If $A$ is a symmetric $(0,2)$-tensor then we define
the $(0,k+2)$-tensor $Q(A,T)$ by
\begin{eqnarray*}
& & Q(A,T)(X_1, \ldots , X_k, X,Y) \ =\
(X \wedge _{A} Y \cdot T)(X_1,\ldots ,X_k)\\
&=&- T((X \wedge _A Y)X_1,X_2,\ldots ,X_k)
- \cdots - T(X_1,\ldots ,X_{k-1},(X \wedge _A Y)X_k) .
\end{eqnarray*}
The tensor $Q(A,T)$
is called the {\sl Tachibana tensor of the tensors} $A$ and $T$,
in short the Tachibana tensor 
(see, e.g., \cite{{DGHHY}, {DGHZ01}, {DGJZ}, {2011-DGPSS}, {DHV2008}, {DeHoJJKunSh}}).
Thus, among other things, we have the $(0,6)$-tensors:
$R \cdot R$, $R \cdot C$, $C \cdot R$, $C \cdot C$, 
$Q(g,R)$, $Q(S,R)$, $Q(g,C)$ and $Q(S,C)$, as well as the $(0,4)$-tensors:
$R \cdot S$, $C \cdot S$ and $Q(g,S)$.  
For a symmetric $(0,2)$-tensors $A$ and $B$ we define their Kulkarni-Nomizu product $A \wedge B$ by 
(see, e.g., \cite{{DGHHY}, {2011-DGPSS}})
\begin{eqnarray*}
(A\wedge B )(X_{1},X_{2},X_{3},X_{4})&=&A(X_{1},X_{4}) B(X_{2},X_{3})+A(X_{2},X_{3}) B(X_{1},X_{4})\\
&-&A(X_{1},X_{3}) B(X_{2},X_{4})
- A(X_{2},X_{4}) B(X_{1},X_{3}).
\end{eqnarray*}
A semi-Riemannian manifold $(M,g)$, $n \geqslant 3$,
is said to be an {\sl Einstein manifold} (see, e.g., \cite{Besse-1987})
if at every point of $M$ its Ricci tensor $S$ is proportional to the metric tensor $g$, 
i.e., on $M$ we have
\begin{eqnarray}
S &=& \frac{\kappa }{n}\, g . 
\label{einstein000}
\end{eqnarray}
According to 
{\cite[p. 432] {Besse-1987}},
(\ref{einstein000}) is called the {\sl Einstein metric condition}.
Einstein manifolds form a natural subclass
of several classes of semi-Riemannian manifolds which are determined by
curvature conditions imposed on their Ricci tensor 
{\cite[Table, pp. 432-433] {Besse-1987}}.
These conditions are named {\sl generalized Einstein curvature conditions}
{\cite[Chapter XVI] {Besse-1987}}.

A semi-Riemannian manifold $(M,g)$, $n \geqslant 3$, 
is locally symmetric if 
\begin{eqnarray}
\nabla R &=& 0
\label{locally-symmetric}
\end{eqnarray}
on $M$ (see, e.g., {\cite[Chapter 1.5] {Lumiste}}). 
Non-reducible locally symmetric manifolds are Einstein manifolds.
The equation (\ref{locally-symmetric})
implies the following integrability condition
${\mathcal{R}}(X,Y ) \cdot R = 0$, or briefly, 
\begin{eqnarray}
R \cdot R &=& 0 .
\label{semisymmetry}
\end{eqnarray}
Semi-Riemannian manifold satisfying (\ref{semisymmetry})
is called {\sl semisymmetric} (see, e.g., 
{\cite[Chapter 8.5.3] {TEC_PJR_2015}}, {\cite[Chapter 20.7] {Chen-2011}}, 
{\cite[Chapter 1.6] {Lumiste}}, \cite{{Szabo}, {LV3-Foreword}}).
Semisymmetric manifolds form a subclass of the class of pseudosymmetric manifolds.
A semi-Riemannian manifold $(M,g)$, $n \geqslant 3$, is said to be {\sl pseudosymmetric} 
if the tensors $R \cdot R$ and $Q(g,R)$ are linearly dependent at every point of $M$ \cite{D-1992}
(see also {\cite[Chapter 8.5.3] {TEC_PJR_2015}}, 
{\cite[Chapter 20.7] {Chen-2011}},
{\cite[Chapter 6] {DHV2008}},
{\cite[Chapter 12.4] {Lumiste}}, 
{\cite[Chapter 7] {MSV}}, 
\cite{{HV_2007}, {HaVerSigma}, {KowSek_1997}, {LV1}, {LV2}, {LV3-Foreword}}). 
This is equivalent to
\begin{eqnarray}
R \cdot R &=& L_{R}\, Q(g,R) 
\label{pseudo}
\end{eqnarray}
on ${\mathcal{U}}_{R} = \{x \in M\, | \, R - (\kappa / (n-1)n )\, G \neq 0\ \mbox {at}\ x \}$,
where $L_{R}$ is some function on this set. 
Examples of non-semisymmetric pseudosymmetric manifolds are presented among others 
in \cite{{1989_DDV}, {DK}, {DVV1991}}. 
Let ${\mathcal U}_{S}$ be the set of all points of a semi-Riemannian manifold $(M,g)$,
$n \geqslant 3$, at which 
$S$ is not proportional to $g$, i.e.,
${\mathcal U}_{S} \, = \,  \{x \in M\, | \, 
S - ( \kappa / n)\, g \neq 0\ \mbox {at}\ x \}$. 
A semi-Riemannian manifold $(M,g)$, $n \geqslant 3$, is called {\sl Ricci-pseudosymmetric} 
if the tensors $R \cdot S$ and $Q(g,S)$ are linearly dependent at every point of $M$
(see, e.g., {\cite[Chapter 8.5.3] {TEC_PJR_2015}}, \cite{{D-1992}, {DGHSaw}, {LV3-Foreword}}).
This is equivalent on ${\mathcal{U}}_{S} \subset M$ to 
\begin{eqnarray}
R \cdot S &=& L_{S}\, Q(g,S) , 
\label{Riccipseudo07}
\end{eqnarray}
where $L_{S}$ is some function on this set. 
Every warped product manifold $\overline{M} \times _{F} \widetilde{N}$
with an $1$-dimensional manifold $(\overline{M}, \overline{g})$  and
an $(n-1)$-dimensional Einstein semi-Riemannian manifold $(\widetilde{N}, \widetilde{g})$, $n \geqslant 3$, 
and a warping function $F$, 
is a Ricci-pseudosymmetric manifold
(see, e.g., {\cite[Chapter 8.5.3] {TEC_PJR_2015}}, {\cite[Section 1] {Ch-DDGP}}, {\cite[Example 4.1] {DGJZ}}, {\cite{DH1}}).
A semi-Riemannian manifold $(M,g)$ is said to be {\sl{pseudosymmetric of constant type}} \cite{{BKV}, {KowSek_1997}, {KowSek_1998}},
resp., {\sl{Ricci-pseudosymmetric of constant type}} \cite{G6}, 
if the function $L_{R}$ is a constant on ${\mathcal{U}}_{R} \subset M$,    
resp., if the function $L_{S}$ is a constant on ${\mathcal{U}}_{S} \subset M$.
Let
${\mathcal U}_{C}$ be the set of all points of a semi-Riemannian manifold $(M,g)$, $n \geqslant 4$, 
at which $C \neq 0$. We note that
${\mathcal{U}}_{S} \cup {\mathcal{U}}_{C} = {\mathcal{U}}_{R}$ (see, e.g., \cite{DGHHY}). 
A semi-Riemannian manifold $(M,g)$, $n \geqslant 4$, is said to have {\sl pseudosymmetric Weyl tensor}
if the tensors $C \cdot C$ and $Q(g,C)$ are linearly dependent at every point of $M$ 
(see, e.g., {\cite[Chapter 20.7] {Chen-2011}}, \cite{{DGHHY}, {DGHSaw}, {DGJZ}}).
This is equivalent on ${\mathcal U}_{C}\subset M$ to 
\begin{eqnarray}
C \cdot C &=& L_{C}\, Q(g,C) ,  
\label{4.3.012}
\end{eqnarray}
where $L_{C}$ is some function on this set. 
Every warped product manifold 
$\overline{M} \times _{F} \widetilde{N}$,
with $\dim \overline{M}  = \dim \widetilde{N} = 2$, 
satisfies (\ref{4.3.012})
(see, e.g., \cite{{DGHHY}, {DGHSaw}, {DGJZ}} and references therein).
Thus in particular,
the Schwarzschild spacetime, the Kottler spacetime
and the Reissner-Nordstr\"{o}m spacetime satisfy (\ref{4.3.012}).
Recently, manifolds satisfying (\ref{4.3.012})
were investigated among others in \cite{{DGHHY}, {DGJZ}, {DeHoJJKunSh}}.
Warped product manifolds $\overline{M} \times _{F} \widetilde{N}$, of dimension $\geqslant 4$,
satisfying on 
${\mathcal U}_{C} \subset \overline{M} \times _{F} \widetilde{N}$ the condition
\begin{eqnarray}
R \cdot R - Q(S,R) &=& L\, Q(g,C) ,  
\label{genpseudo01}
\end{eqnarray}
where $L$ is some function on this set,
were studied among others in \cite{{49}, {DGJZ}}.
For instance, 
in \cite{49} necessary and sufficient conditions for  
$\overline{M} \times _{F} \widetilde{N}$ to be a manifold satisfying (\ref{genpseudo01}) are given.
Moreover, in that paper it was proved that 
any $4$-dimensional warped product manifold $\overline{M} \times _{F} \widetilde{N}$, 
with an $1$-dimensional base $(\overline{M},\overline{g})$, 
satisfies (\ref{genpseudo01}) {\cite[Theorem 4.1] {49}}.
The warped product manifold $\overline{M} \times _{F} \widetilde{N}$, 
with $2$-dimensional base $(\overline{M},\overline{g})$ 
and $(n-2)$-dimensional space of constant curvature $(\widetilde{N},\widetilde{g})$, $n \geqslant 4$,
is a manifold satisfying 
(\ref{4.3.012}) and (\ref{genpseudo01}) {\cite[Theorem 7.1 (i)] {DGJZ}}.
We refer to
\cite{{Ch-DDGP}, {D-1992}, {DGHHY}, {DGHSaw}, {DGHZ01}, {DGJZ}, {DHV2008}, {DeHoJJKunSh}, {DP-TVZ}, {SDHJK}, {LV1}} 
for details on semi-Riemannian manifolds satisfying (\ref{pseudo}) and (\ref{Riccipseudo07})-(\ref{genpseudo01}), 
as well as other conditions of this kind, named 
{\sl{pseudosymmetry type curvature conditions}}
or 
{\sl{pseudosymmetry type conditions}}. 
It seems that (\ref{pseudo}) 
is the most important condition of that family of curvature conditions (see, e.g., \cite{DGJZ}).
We also can state that
the Schwarzschild spacetime, the Kottler spacetime, the Reissner-Nordstr\"{o}m spacetime, 
as well as the Friedmann-Lema{\^{\i}}tre-Robertson-Walker spacetimes are the "oldest" examples 
of pseudosymmetric warped product manifolds (see, e.g., \cite{{DGJZ}, {DHV2008}, {DVV1991}, {SDHJK}}).

Investigations on semi-Riemannian manifolds $(M,g)$, $n \geqslant 4$,
satisfying 
(\ref{pseudo}) and (\ref{4.3.012})
or
(\ref{pseudo}) and (\ref{genpseudo01})
on ${\mathcal{U}}_{S} \cap {\mathcal{U}}_{C} \subset M$
lead to the following condition ({\cite[Theorem 3.2 (ii)] {DY1994}}, {\cite[Lemma 4.1] {P43}}, see also
{\cite[Section 1] {DGJZ}})
\begin{eqnarray}
R &=& \frac{\phi}{2}\, S\wedge S + \mu\, g\wedge S + \frac{\eta}{2}\, g \wedge g ,
\label{eq:h7a}
\end{eqnarray}
where 
$\phi$, $\mu $ and $\eta $ are some functions on ${\mathcal U}_{S} \cap {\mathcal U}_{C}$.
We note that if (\ref{eq:h7a}) is satisfied 
at a point of ${\mathcal U}_{S} \cap {\mathcal U}_{C}$ then at this point we have
$\mathrm{rank} ( S - \alpha \, g ) > 1$ for any $\alpha \in  {\mathbb{R}}$.
A semi-Riemannian manifold $(M,g)$, $n \geqslant 4$, satisfying (\ref{eq:h7a}) on 
${\mathcal U}_{S} \cap {\mathcal U}_{C} \subset M$ 
is called a {\sl Roter type manifold}, or a {\sl Roter type space}, or a {\sl Roter space} 
\cite{{P106}, {DGP-TV01}, {DGP-TV02}}. 

Curvature properties of $2$-recurrent semi-Riemannian manifolds ($\nabla ^{2} R = R \otimes \psi$)
were investigated by Professor Witold Roter among others in \cite{Roter-1964}.
In that paper it was shown that 
\begin{eqnarray}
R 
&=& 
\frac{1}{2 \kappa}\, S \wedge S 
\label{Roter001}
\end{eqnarray}
holds on some $2$-recurrent manifolds {\cite[Theorem 1] {Roter-1964}}.
It seems that \cite{Roter-1964} is the first paper on manifolds satisfying (\ref{Roter001}).
Evidently, (\ref{Roter001}) is a special case of (\ref{eq:h7a}) ($\mu = \eta = 0$), i.e.
\begin{eqnarray}
R 
&=& 
\frac{\phi}{2 }\, S \wedge S .
\label{Roter002}
\end{eqnarray}
We refer to 
{\cite[Example 3.1] {Glog-2005}},  
{\cite[Section 4] {Kow01}} 
and 
{\cite[Example 3.1] {MADEO-2001}}
for results on manifolds satisfying (\ref{Roter002}).

Curvature properties of semi-Riemannian manifolds of dimension $\geqslant 4$
with parallel Weyl conformal curvature tensor ($\nabla C = 0$),
non-conformally flat ($C \neq 0$) and non-locally symmetric ($\nabla R \neq 0$), 
were investigated among others in \cite{Derdz-Roter-1980}.
Such manifolds are also named essentially conformally symmetric manifolds, 
e.c.s. manifolds, in short. 
In \cite{Derdz-Roter-1980} it was shown that the Weyl tensor $C$ 
of some e.c.s. manifolds is of the form  
$C = (\phi/2)\, S \wedge S$.
Since the scalar curvature $\kappa$ of every e.c.s. manifold vanishes, the last equation yields 
$R=(\phi/2)\, S \wedge S + (1/(n-2))\, g \wedge S\,$.
Thus we have (\ref{eq:h7a}) with $\mu =1/(n-2)$ and $\eta = 0$. 

Roter type manifolds and in particular Roter type hypersurfaces
(i.e. hypersurfaces satisfying (\ref{eq:h7a})), 
in semi-Riemannian spaces of constant curvature were studied in:
\cite{{P106}, {DGHHY}, {DGHZ01}, {DGP-TV01}, {R102}, {DK}, {DP-TVZ}, {DePlaScher}, {DeScher}, {G5}, 
{Kow02}}. 
Roter type manifolds satisfy several pseudosymmetry type curvature conditions, we have
\begin{theo} \cite{{DGHSaw}, {G5}}
If $(M,g)$, $n \geqslant 4$, is a semi-Riemannian Roter space satisfying
(\ref{eq:h7a}) on ${\mathcal U}_{S} \cap {\mathcal U}_{C} \subset M$ 
then on this set we have
\begin{eqnarray}
\label{S2}
S^{2} &=& \alpha _{1}\, S + \alpha _{2} \, g ,
\ \ \
\alpha _{1} 
\ =\ 
\kappa + \frac{(n-2)\mu -1 }{\phi} ,
\ \ \
\alpha _{2}
\ =\
\frac{\mu \kappa + (n-1) \eta }{\phi } ,\\
\label{LR} 
R \cdot R &=& L_{R}\, Q(g,R),
\ \ \
L_{R}
\ =\
\frac{1}{\phi} \left( (n-2) (\mu ^{2} - \phi \eta) - \mu \right) ,
\end{eqnarray}
\begin{eqnarray*}
R \cdot C &=& L_{R}\, Q(g,C),
\ \ \
R \cdot S \ =\ L_{R}\, Q(g,S),
\end{eqnarray*}
\begin{eqnarray}
\label{LSR}
R \cdot R &=& Q(S,R) + L \, Q(g,C) ,
\ \ \
L
\ =\
L_{R} + \frac{\mu }{\phi }
\ =\
\frac{n-2}{\phi} (\mu ^{2} - \phi \eta),\\
\label{LCC}
C \cdot C &=& L_{C}\, Q(g,C) ,
\ \ \ 
L_{C} 
\ =\ L_{R} + \frac{1}{n-2} (\frac{\kappa }{n-1} - \alpha _{1} ) ,
\end{eqnarray}
\begin{eqnarray*}
C \cdot R &=& L_{C}\, Q(g,R), \ \ \  
C \cdot S \ =\ L_{C}\, Q(g,S) ,\\
R \cdot C - C \cdot R &=& 
\left( \frac{1}{\phi} ( \mu - \frac{1}{n-2} ) + \frac{\kappa }{n-1} \right) Q(g,R)
+ \left( \frac{\mu}{\phi } ( \mu - \frac{1}{n-2}) - \eta \right) Q(S,G) ,
\end{eqnarray*}
\begin{eqnarray*}
C \cdot R - R \cdot C  &=&   Q(S,C) - \frac{\kappa}{ n-1 }\, Q(g,C) . 
\end{eqnarray*}
\end{theo} 

\begin{rem}
(i) In the standard Schwarzschild coordinates $(t; r; \theta; \phi)$, 
and the physical units ($c = G = 1$), the Reissner-Nordstr\"{o}m-de Sitter ($\Lambda > 0$), 
and the Reissner-Nordstr\"{o}m-anti-de Sitter ($\Lambda < 0$) metrics are given by the line element (see, e.g., \cite{SH})
\begin{eqnarray}
\ \ \ \ \ \ \ ds^{2} &=& - h(r)\, dt^{2} + h(r)^{-1}\, dr^{2} + r^{2}\, ( d\theta^{2} + \sin ^{2}\theta \, d\phi^{2}),\ \
h(r) \ =\ 1 - \frac{2M}{r} + \frac{Q^{2}}{r^{2}} - \frac{\Lambda r^{2}}{3}  ,
\label{rns01}
\end{eqnarray}
where $M$, $Q$ and $\Lambda$ are non-zero constants. 
\newline
(ii) {\cite[Section 6] {DGHJZ01}} 
The metric (\ref{rns01}) satisfies (\ref{eq:h7a}) with 
\begin{eqnarray*}
\phi &=&
\frac{3}{2}\, ( Q^{2}  -  M r) r^{4} Q^{-4} ,\ \ \ \mu \ =\
\frac{1}{2}\, ( Q^{4} + 3 Q^{2} \Lambda r^{4} - 3 \Lambda M r^{5} ) Q^{-4} ,\\
\eta
&=&
\frac{1}{12}\, ( 3 Q^{6}
+ 4 Q^{4} \Lambda r^{4}
- 3 Q^{4} M r
+ 9 Q^{2} \Lambda ^{2} r^{8}
- 9 \Lambda^{2} M r^{9} ) r^{-4} Q^{-4}   .
\end{eqnarray*}
If we set $\Lambda = 0$ in (\ref{rns01})  
then we obtain the line element of the Reissner-Nordstr\"{o}m spacetime, 
see, e.g., {\cite[Section 9.2] {GrifPod}} and references therein. 
It seems that the Reissner-Nordstr\"{o}m spacetime is the "oldest" example 
of the Roter type warped product manifold.
\newline
(iii) 
Some comments on pseudosymmetric manifolds (also called Deszcz symmetric spaces),
as well as Roter spaces, are given in {\cite[Section 1] {DecuP-TSVer}}:
"{\sl{From a geometric point of view, the Deszcz symmetric spaces may well
be considered to be the simplest Riemannian manifolds next to the real space forms.}}" 
and 
"{\sl{From an algebraic point of view, Roter spaces may well be considered to
be the simplest Riemannian manifolds next to the real space forms.}}"
For further comments we refer to \cite{LV3-Foreword}.
\end{rem}
A semi-Riemannian manifold $(M,g)$, $n \geqslant 3$, is said to be 
a {\sl quasi-Einstein manifold} if 
\begin{eqnarray}
\mathrm{rank}\, (S - \alpha\, g) &=& 1
\label{quasi02}
\end{eqnarray}
on ${\mathcal U}_{S} \subset M$, where $\alpha $ is some function on this set.
Quasi-Einstein manifolds arose during the study of exact solutions
of the Einstein field equations and the investigation on quasi-umbilical hypersurfaces 
of conformally flat spaces (see, e.g., \cite{{DGHSaw}, {DGJZ}} and references therein). 
Quasi-Einstein manifolds satisfying some pseudosymmetry type conditions  
were investigated among others in
\cite{{P119}, {Ch-DDGP}, {DGHHY}, {DGHZ01}, {DeHoJJKunSh}}. 
Quasi-Einstein hypersurfaces in semi-Riemannian spaces of constant curvature
were studied among others in
\cite{{DGHS}, {R102}, {DHS105}, {G6}}, see also {\cite[Chapter 6.2] {TEC_PJR_2015}},
{\cite[Chapter 19.5] {Chen-2011}}, {\cite[Chapter 4.6] {Chen-2017}},
\cite{{DGHSaw}, {LV3-Foreword}} and references therein. 
We mention that there are different extensions of the class of quasi-Einstein manifolds. 
For instance we have the class of almost quasi-Einstein manifolds \cite{Chen-2017-KJM}
as well as the class of $2$-quasi-Einstein manifolds (see, e.g. \cite{{2016_DGHZhyper}, {DGJZ}}).

\section{Geodesic mappings}

Let $(M,g)$ and $(\overline{M},\overline{g})$ be two $n$-dimensional semi-Riemannian manifolds and let
a diffeomorphism $h:M\to\overline{M}$ be a geodesic mapping.
It is known that in a common coordinate system $\{x^1,\ldots,x^n\}$, the Christoffel symbols, the curvature tensors and the Ricci tensors
of $(M,g)$ and $(\overline{M},\overline{g})$ are related by (see \cite{Sinj}, {\cite[Chapter 8]{MSV}})
\begin{eqnarray}
\label{gama}
\overline\Gamma^h_{ij} &=& \Gamma^h_{ij}+\delta^h_i\psi_j+\delta^h_j\psi_i\,,\\
\nonumber
\overline{R}^{h}_{\ ijk} &=& {R}^{h}_{\ ijk}+\delta^h_j\psi_{ik}-\delta^h_k\psi_{ij}\,,\\
\label{S}
\overline{S}_{ij} &=& S_{ij}-(n-1)\psi_{ij}\,,
\end{eqnarray}
where
\begin{equation}
      \label{psiij}
\psi_{ij} \ =\ \nabla_j\psi_i-\psi_i\psi_j\,,\ \ \psi_i \ =\ \frac{1}{2(n+1)}\frac{\partial}{\partial x^i}
\left(\log\left|\frac{\det \overline{g}}{\det g}\right|\right)\,.
\end{equation}
We will denote by
$h:(M,g) \ddd (\overline{M},\overline{g})$
a geodesic mapping of $(M,g)$ onto $(\overline{M},\overline{g})$ 
and the manifolds $(M,g)$ and $(\overline{M},\overline{g})$
will be called {\it geodesically related}. Further, a geodesic mapping $h:(M,g) \ddd (\overline{M},\overline{g})$ 
is called {\it non-trivial} on $M$ if the covector field $\psi$ with the local components $\psi_i$ is non-zero.
It is also known that a manifold $(M,g)$ can be geodesically mapped into $(\overline{M},\overline{g})$ if and only if
there exists a covector field $\psi$ on $M$ which is a gradient with the property that
\begin{equation}
\label{geo}
\nabla_k\overline{g}_{ij}=2\psi_k\overline{g}_{ij}+\psi_i\overline{g}_{jk}+\psi_j\overline{g}_{ik}\,.
\end{equation}
We have the following theorem.
\begin{theo}\cite{{DD}, {DH}}
If $(M,g)$ is a pseudosymmetric semi-Riemannian manifold
admitting a non-trivial geodesic mapping $h$
onto a manifold $(\overline{M},\overline{g})$ 
then $(\overline{M},\overline{g})$ is also a pseudosymmetric manifold.
Moreover,
$$\psi_{ij}=L_Rg_{ij}-L_{\overline{R}}\,\overline{g}_{ij}\,,$$
and $L_R=constant$ if and only if $L_{\overline{R}}=constant$.
\end{theo}

It is worth to noticing that the above statement was presented in the survey paper \cite{M}, but without proof.

In the paper \cite{DD_PMH} was considered manifolds satisfying $R\cdot R=LQ(S,R)$ or $R\cdot R=Q(S,R)$ admitting geodesic
mappings. 

Let $M$ be a 2-dimensional manifold with the metric
$$ds^2=\mf{a}(x)dx^2+\mf{b}(x)dy^2\,.$$
It is known (\cite{{Hint1}, {Hint2}}, see also {\cite[p. 356] {MSV}) that $M$ maps geodesically into $\overline{M}$ with the metric
$$d\overline{s}^2=\frac{p\mf{a}(x)}{(1+q\mf{b}(x))^2}\,dx^2+\frac{p\mf{b}(x)}{1+q\mf{b}(x)}\,dy^2\,,$$
where $p\neq 0$ and $q$ are real parameters, $x$ and $y$ are common coordinates. Evidently we assume that $\mf{a}(x)\neq 0$,
$\mf{b}(x)\neq 0$ and $1+q\mf{b}(x)\neq 0$.\\ Taking into account that
$$g_{11}=\mf{a}(x),\ g_{22}=\mf{b}(x),\ g_{12}=0,\ \overline{g}_{11}=\frac{p\mf{a}(x)}{(1+q\mf{b}(x))^2},\ \overline{g}_{22}=\frac{p\mf{b}(x)}{1+q\mf{b}(x)},\ \overline{g}_{12}=0$$
it is easy to see that the only non-zero components of Christoffel symbols are the following
\begin{equation}
\label{Ch2}
\Gamma^1_{11}=\frac{\mf{a}'}{2\mf{a}},\  \Gamma^2_{12}=\frac{\mf{b}'}{2\mf{b}},\ \Gamma^1_{22}=-\frac{\mf{b}'}{2\mf{a}},
\end{equation}
where $\mf{a}'=\frac{d\mf{a}}{dx},\ \mf{b}'=\frac{d\mf{b}}{dx}$.
Moreover, the equality (\ref{geo}) is satisfied with $\psi$ given by
\begin{equation}
\label{psi}
\psi_1=-\frac 12\,\frac{q\mf{b}'}{1+q\mf{b}}\,,\ \psi_2=0\,.
\end{equation}
Since we are interested in non-trivial geodesic mappings, throughout this paper, we moreover assume that $q\neq 0$ and $\mf{b}'(x)\neq 0$.

The following lemma is useful.
\begin{lem}
Let the metric $g$ on $\mathbb R^2=\{(x,y):x,y\in \mathbb R\}$ be of the form $g_{11}=\mf{a}(x),\ g_{22}=\mf{b}(x),\ g_{12}=0$.
For the Gauss curvature $\kappa_G$ of the metric $g$ we have $\kappa_G=K=const.$ if and only if
\begin{equation}
\label{kG}
\mf{a}=\frac{(\mf{b}')^2}{\mf{b}(E-4K\mf{b})}\,,\ E\in\mathbb R\,.
\end{equation}
\end{lem}
\begin{proof}
We have (c.f. (\ref{RCh})
$$R_{1221}=g_{11}(\partial_1\Gamma^1_{22}-\partial_2\Gamma^1_{21}+\Gamma^r_{22}\Gamma^1_{r1}-\Gamma^r_{21}\Gamma^1_{r2})$$
and, in virtue of (\ref{Ch2}) we obtain
$$R_{1221}=\frac 12(-\mf{b}'' +\frac{\mf{a}'\mf{b}'}{2\mf{a}}+\frac{(\mf{b}')^2}{2\mf{b}})\,.$$
On the other hand we have 
$$R_{1221}=\frac{\kappa}{2}(g_{11}g_{22}-g^2_{12}),\ {\rm i.e.}\ \ R_{1221}=\frac{\kappa}{2}\mf{a}\mf{b}=\kappa_G\mf{a}\mf{b}\,.$$
Thus we get
$$\kappa_G=-\frac{2\mf{a}\mf{b}\mf{b}''-\mf{b}\mf{a}'\mf{b}'-\mf{a}(\mf{b}')^2}{4(\mf{ab})^2}$$
and by our assumption
$$2\mf{abb}''-\mf{ba}'\mf{b}'-\mf{a}(\mf{b}')^2\ =\ -4K\mf{a}^2\mf{b}^2.$$
So we obtain the following Bernoulli's equation with respect to the unknown function $\mf{a}$
$$\mf{a}' +  \left( \frac{\mf{b}'}{\mf{b}} - \frac{2\mf{b}''}{\mf{b}'} \right) \mf{a} 
= \frac{4K\mf{b}}{\mf{b}'}\mf{a}^2\,.$$
Thus standard calculation leads to the solution of the form (\ref{kG}).
\end{proof}

\section{Warped product manifolds}
Let $(\widehat{M},\widehat{g})$ and $(\widetilde{N},\widetilde{g})$,
$\dim \widehat{M} = p$, $\dim\widetilde{N} = n-p$, $1 \leqslant p < n$, 
be semi-Riemannian manifolds
and $F$ a positive smooth function on $\widehat{M}$. 
The warped product $\widehat{M} \times _F \widetilde{N}$ 
of $(\widehat{M},\widehat{g})$ and $(\widetilde{N}, \widetilde{g})$ 
is the product manifold $\widehat{M} \times \widetilde{N}$ 
with the metric tensor $g$ defined by
\begin{eqnarray*}
g &=& \widehat{g} \times _F \widetilde{g} =
{\pi}_1^{*} \widehat{g} + (F \circ {\pi}_1)\, {\pi}_2^{*} \widetilde{g} ,
\end{eqnarray*}
where 
${\pi}_1 : \widehat{M} \times \widetilde{N} \longrightarrow \widehat{M}$ and 
${\pi}_2 : \widehat{M} \times \widetilde{N} \longrightarrow \widetilde{N}$ 
are the natural projections on $\widehat{M}$ and $\widetilde{N}$, respectively 
(see, e.g., \cite{ON} and references therein).
Let $(\widehat{M},\widehat{g})$ and $(\widetilde{N},\widetilde{g})$
be covered by systems of charts $\{ U;x^{a} \}$ and 
$\{ V;y^{\alpha } \} $,
respectively and let 
$ \{ U \times V ; x^{1}, \ldots ,x^{p},x^{p+1} =  y^{1}, \ldots , x^{n} = y^{n-p} \} $ 
be a product chart for $\widehat{M} \times \widetilde{N}$. 
The local components $g_{ij}$ of the metric $g = \widehat{g} \times _F \widetilde{g}$ with respect
to this chart are the following
$g_{ij} = \widehat{g}_{ab}$ if $i = a$ and $j = b$,
$g_{ij} = F\, \widetilde{g}_{\alpha \beta }$ if $i = \alpha $ and $j = \beta $, and
$g_{ij} = 0$ otherwise,
where $a,b,c, d, f \in \{ 1, \ldots ,p \} $, 
$ \alpha , \beta , \gamma , \delta \in \{ p+1, \ldots ,n \} $
and $h, i, j, k, l, m, r, s \in \{ 1,2, \ldots ,n \} $. We will denote by hats
(resp., by tildes) tensors formed from $\widehat{g}$ (resp., $\widetilde{g}$).
The local components 
\begin{eqnarray*}
\Gamma ^{h} _{ij} \ =\ \frac{1}{2}\, g^{hs} ( \partial_{i} g_{js} + \partial_{j} g_{is} - \partial_{s} g_{ij}),
\ \ 
\partial _j \ =\ \frac{\partial }{\partial x^{j}} ,
\end{eqnarray*} 
of the Levi-Civita connection $\nabla $
of $\widehat{M} \times _F \widetilde{N}$ are the following (see, e.g., \cite{DGJZ})
\begin{eqnarray}
\label{Chp}
& &
\Gamma ^{a} _{bc}\ =\ \widehat{\Gamma } ^{a} _{bc} , \ \
\Gamma ^{\alpha } _{\beta \gamma } 
\ =\ \widetilde{\Gamma } ^{\alpha } _{\beta \gamma } ,  \ \
\Gamma ^{a} _{\alpha \beta } 
\ =\ - \frac{1}{2} \widehat{g} ^{ab} F_b \widetilde{g} _{\alpha \beta } , \\
\nonumber
& &
\Gamma ^{\alpha } _{a \beta } 
\ =\ \frac{1}{2F} F_a \delta ^{\alpha } _{\beta } , \ \
\Gamma ^{a} _{\alpha b}\ =\ \Gamma ^{\alpha } _{ab}\, =\, 0 , \ \
F_a\ =\ \frac{\partial F}{\partial x^{a}} .
\end{eqnarray}
The local components
\begin{equation}
R_{hijk}\ =\ g_{hs}R^{s}_{\, ijk}\ =\ g_{hs} (\partial _k \Gamma ^{s} _{ij} 
- \partial _j \Gamma ^{s} _{ik} + \Gamma ^{r} _{ij} \Gamma ^{s} _{rk}
- \Gamma ^{r} _{ik} \Gamma ^{s} _{rj} )
\label{RCh}
\end{equation} 
of the Riemann-Christoffel curvature tensor $R$
and the local components $S_{ij}$ of the Ricci tensor $S$
of the warped product $\widehat{M} \times _F \widetilde{N}$ which may not vanish
identically are the following:
\begin{eqnarray}
\ \ \ \ \ \
\nonumber
R_{abcd} &=& \widehat{R}_{abcd} ,\ 
R_{\alpha ab \delta} \ =\
- \frac{1}{2}\, T_{ab} \widetilde{g}_{\alpha \delta} ,\ 
R_{\alpha \beta \gamma \delta} \ =\
F \widetilde{R}_{\alpha \beta \gamma \beta} 
- \frac{\Delta_1 F}{4}\, \widetilde{G}_{\alpha \beta \gamma \delta}\, ,\\
\ \ \ \ \ \
S_{ab} &=& \widehat{S}_{ab} - \frac{n-p}{2 F}\, T_{ab} ,\ 
S_{\alpha \beta } \ =\  
\widetilde{S}_{\alpha \beta } 
- \frac{1}{2}\, ( \mathrm{tr}(T) + \frac{n-p-1}{2F} \Delta _1 F )\, \widetilde{g}_{\alpha \beta } ,
\label{AL2}\\
\ \ \ \ \ \
T_{ab}&=& 
\widehat{\nabla }_a F_b - \frac{1}{2F} F_a F_b ,\ 
\mathrm{tr}(T) \ =\ \widehat{g}^{ab} T_{ab} , \  
\Delta _{1} F \ =\ {\Delta}_{1\, \widehat{g}} F\ =\ 
\widehat{g}^{ab} F_a F_b ,
\label{AL3}
\end{eqnarray}
where $T$ is the $(0,2)$-tensor with the local components $T_{ab}$.
The scalar curvature $\kappa $ of
$\widehat{M} \times _F \widetilde{N}$ satisfies the following equation
$$\kappa = \widehat{\kappa } + \frac{1}{F}\, \widetilde{\kappa }
- \frac{n - p}{F}\, ( \mathrm{tr}(T) + \frac{n - p -1 }{4F} \Delta _1 F).$$

Let $(\widehat{M},\widehat{g})$ be a $2$-dimensional manifold with a metric $\widehat{g}$ given by
$$\widehat{g}_{11}=\mf{a}(x^1)\,,\ \widehat{g}_{22}=\mf{b}(x^1)\,,\ \widehat{g}_{12}=0$$
and $(\widetilde{N},\widetilde{g})$ be an $(n-2)$-dimensional, $n\geqslant 4$, semi-Riemannian space of constant curvature,
when $n\geqslant 5$. Next let $\widehat{M} \times _F \widetilde{N}$ be the warped product with warping function $F=F(x^1,x^2)$.
Let $(\overline{\widehat{M}},\overline{\widehat{g}})$ be a manifold geodesically related to $(\widehat{M},\widehat{g})$ with a metric
$\overline{\widehat{g}}$ given by
$$\overline{\widehat{g}}_{11}
=\frac{p\mf{a}(x^1)}{(1+q\mf{b}(x^1))^2}\,,\ \overline{\widehat{g}}_{22}=\frac{p\mf{b}(x^1)}{1+q\mf{b}(x^1)}\,,\ \overline{\widehat{g}}_{12}=0$$
and a covector field $\psi$ such as in (\ref{psi}).

We will find the necessary and sufficient conditions that the warped product manifold 
$\widehat{M} \times _F \widetilde{N}$ can be geodesically mapped
into the warped product manifold $\overline{\widehat{M}} \times _{\overline{F}} \widetilde{N}$ with a warping function $\overline{F}=\overline{F}(x^1,x^2)$.
Under our assumptions we have
\begin{eqnarray}
\label{55}
g_{11}&=&\mf{a}(x^1)\,,\ g_{22}=\mf{b}(x^1)\,,\ g_{\alpha\beta}=F\widetilde{g}_{\alpha\beta}\,,\\
\nonumber
\overline{g}_{11}&=&\frac{p\mf{a}(x^1)}{(1+q\mf{b}(x^1))^2}\,,\ \overline{g}_{22}=\frac{p\mf{b}(x^1)}{1+q\mf{b}(x^1)}\,,\ \overline{g}_{\alpha\beta}=
\overline{F}\widetilde{g}_{\alpha\beta}\,,
\end{eqnarray}
and remaining components of $g$ and $\overline{g}$ vanish.\\
It is obvious that the equality (\ref{geo}) is satisfied for $i=a, j=b, k=c$.\\
Considering the case $i=a, j=\alpha, k=\beta$ we have, in virtue of (\ref{Chp})
$$\frac{\partial \overline{g}_{a\alpha}}{\partial x^\beta}-\Gamma^s_{\beta a}\overline{g}_{s\alpha}-\Gamma^s_{\beta\alpha}\overline{g}_{sa}=
2\psi_\beta \overline{g}_{a\alpha}+\psi_a \overline{g}_{\alpha\beta}+\psi_\alpha \overline{g}_{a\beta}\,,$$
$$-\Gamma^\epsilon_{\beta a}\overline{g}_{\epsilon\alpha}-\Gamma^\epsilon_{\beta\alpha}\overline{g}_{\epsilon a}-\Gamma^c_{\beta\alpha}\overline{g}_{ca}
=\psi_a \overline{g}_{\alpha\beta}\,,$$
$$-\frac{1}{2F}\,F_a \overline{g}_{\alpha\beta}+\frac 12\frac{F^c}{F}\overline{g}_{ca}F\widetilde{g}_{\alpha\beta}=\psi_a \overline{g}_{\alpha\beta}$$
and finally
\begin{equation}
\label{r4}
-\frac{\overline{F}}{2F}\,F_a+\frac 12F^c \overline{g}_{ca}=\overline{F}\psi_a\,.
\end{equation}
Now, let $i=\alpha, j=\beta, k=a$. We have in sequence
$$\frac{\partial \overline{g}_{\alpha\beta}}{\partial x^a}-\Gamma^s_{a\alpha}\overline{g}_{s\beta}-\Gamma^s_{a\beta}\overline{g}_{s\alpha}=
2\psi_a \overline{g}_{\alpha\beta}+\psi_\alpha \overline{g}_{a\beta}+\psi_\beta \overline{g}_{a\alpha}\,,$$
$$\frac{\partial\overline{F}}{\partial x^a}\widetilde{g}_{\alpha\beta}-\Gamma^\epsilon_{a\alpha}\overline{g}_{\epsilon\beta}-
\Gamma^\epsilon_{a\beta}\overline{g}_{\epsilon\alpha}=2\psi_a \overline{g}_{\alpha\beta}\,,$$
$$\frac{\partial\overline{F}}{\partial x^a}-\frac{1}{2F}F_a \overline{F}-\frac{1}{2F}F_a \overline{F}=2\psi_a \overline{F}\,,$$
$$\frac{1}{\overline{F}}\frac{\partial\overline{F}}{\partial x^a}=\frac{F_a}{F}+2\psi_a\,,\ \ 
\frac{\partial\log{\overline{F}}}{\partial x^a}-
\frac{\partial\log{F}}{\partial x^a}=2\psi_a\,,$$     
and finally
\begin{equation}
\label{r5}
\frac{\partial}{\partial x^a}\left(\log {\frac{\overline{F}}{F}}\right)=2\psi_a\,.
\end{equation}
It is easy to check that in the remaining cases (\ref{geo}) is also satisfied. Thus we have proved
\begin{prop}
Let $(\widehat{M},\widehat{g})$ be a $2$-dimensional manifold with a metric $\widehat{g}$ given by
$$\widehat{g}_{11}=\mf{a}(x^1)\,,\ \widehat{g}_{22}=\mf{b}(x^1)\,,\ \widehat{g}_{12}=0$$
and $(\widetilde{N},\widetilde{g})$ be an $(n-2)$-dimensional, $n\geqslant 4$, semi-Riemannian space of constant curvature,
when $n\geqslant 5$. Next let $\widehat{M} \times _F \widetilde{N}$ be the warped product manifold with warping function $F=F(x^1,x^2)$ and
let $(\overline{\widehat{M}},\overline{\widehat{g}})$ be a manifold geodesically related to $(\widehat{M},\widehat{g})$ with a metric
$\overline{\widehat{g}}$ given by
$$\overline{\widehat{g}}_{11}=\frac{p\mf{a}(x^1)}{(1+q\mf{b}(x^1))^2}\,,\ \overline{\widehat{g}}_{22}=\frac{p\mf{b}(x^1)}{1+q\mf{b}(x^1)}\,,\ \overline{\widehat{g}}_{12}=0$$
and a covector field $\psi$ such as in (\ref{psi}).
Then the warped product manifold $\widehat{M} \times _F \widetilde{N}$ can be geodesically mapped
into the warped product manifold $\overline{\widehat{M}} \times _{\overline{F}} \widetilde{N}$ with a warping function $\overline{F}=\overline{F}(x^1,x^2)$
if and only if the equalities (\ref{r4}) and (\ref{r5}) are satisfied.
\end{prop}
According to {\cite[Theorem 5.3]{DGJZ}} the warped product manifold $\widehat{M} \times _F \widetilde{N}$ with $2$-dimensional manifold 
$(\widehat{M},\widehat{g})$ and $(n-2)$-dimensional semi-Riemannian space of constant curvature is pseudosymmetric on the set
$U_S\cap U_C$ if and only if $T_{ab}$ is proportional to $\widehat{g}_{ab}$ on this set. Therefore let the warped product manifold 
$\widehat{M} \times _F \widetilde{N}$ be such as in Proposition 4.1 with
\begin{equation}
\label{Ff}
F=F(x^1,x^2)=f^2(x^1,x^2)
\end{equation}
and we consider now the condition: $T=\lambda \widehat{g}$.
In view of (\ref{55}) this condition is equivalent to
\begin{equation}
\label{Tg}
(i)\ T_{12}=0\,,\ \ (ii)\ \mf{b}T_{11}=\mf{a}T_{22}\,.
\end{equation}
Further, by (\ref{AL3}) 
$$T_{12}\ =\ 
\widehat{\nabla}_1F_2-\frac{1}{2F}F_1F_2=\partial_1 F_2-F_s\Gamma^s_{12}-\frac{1}{2F}F_1F_2=\partial_1 F_2-F_2\Gamma^2_{12} - \frac{1}{2F}F_1F_2,$$
so using (\ref{Tg})(i), (\ref{Ch2}) and (\ref{Ff}) we get
$$2(f_1f_2+f f_{12})-2ff_2\frac{\mf{b}'}{2\mf{b}}-\frac{1}{2f^2}\,2ff_12ff_2=0\,,$$
\begin{equation}
\label{t1}
f_{12}=f_2\,\frac{\mf{b}'}{2\mf{b}}\,.
\end{equation}
Similarly,
$$T_{11}=\partial_1 F_1-\frac{\mf{a}'}{2\mf{a}}\,F_1-\frac{1}{2F}\,F_1^2=2(f_1^2+ff_{11})-\frac{\mf{a}'}{2\mf{a}}\,2ff_1-\frac{1}{2f^2}4f^2f_1^2=2ff_{11}-
2ff_1\frac{\mf{a}'}{2\mf{a}}\,,$$
$$T_{22}=\partial_2 F_2+\frac{\mf{b}'}{2\mf{a}}\,F_1-\frac{1}{2F}\,F_2^2=2(f_2^2+ff_{22})+\frac{\mf{b}'}{2\mf{a}}\,2ff_1-2f_2^2=2ff_{22}+2ff_1\frac{\mf{b}'}{2\mf{a}}\,.$$
Thus (\ref{Tg})(ii) leads to
$$\mf{b}f_{11}-\mf{a}f_{22}=f_1\left(\frac{\mf{ba}'}{2\mf{a}}+\frac{\mf{b}'\mf{a}}{2\mf{a}}\right)=\frac{f_1}{2\mf{a}}(\mf{ab})'\,,$$
\begin{equation}
\label{t2}
f_{11}-\frac{\mf{a}}{\mf{b}} f_{22}=\frac{f_1}{2\mf{a}\mf{b}}(\mf{a}\mf{b})'\,.
\end{equation}
Now we will find conditions for which equations (\ref{r4}) and (\ref{r5}) will be satisfied.\\
Using (\ref{r4}) for $a=2$, in virtue of (\ref{psi}), we have
$$\frac{\overline{F}}{F}\,F_2=F^c\overline{g}_{c2}=F_s\widehat{g}^{2s}\overline{g}_{22}=F_2\frac{1}{\mf{b}}\frac{p\mf{b}}{1+q\mf{b}}=\frac{pF_2}{1+q\mf{b}}\,,$$
\begin{equation}
\label{FF}
\frac{\overline{F}}{F}=\frac{p}{1+q\mf{b}}\,.
\end{equation}
This implies that $\frac{\partial}{\partial x^2}(\log{\frac{\overline{F}}{F}})=0$. On the other hand
$\frac{\partial}{\partial x^1}(\log\frac{\overline{F}}{F})=-\frac{q\mf{b}'}{1+q\mf{b}}=2\psi_1$. Thus the equality (\ref{r5}) is satisfied.\\
The condition (\ref{r4}) for $a=1$ takes the form 
$-\frac{\overline{F}}{F}F_1+F^c\overline{g}_{c1} = 2 \overline{F} \psi_1$.  Since
$$F^c\overline{g}_{c1}=F_s\widehat{g}^{s1}\overline{g}_{11}=F_1\frac{1}{\mf{a}} \frac{p\mf{a}}{(1+q\mf{b})^2}=F_1\frac{p}{(1+q\mf{b})^2}\,,$$
so using (\ref{FF}) we get
$\mf{b}'F=\mf{b}F_1$ which in term of $f$ takes the form $2\mf{b}f_1=\mf{b}'f$. Applying this equality to (\ref{t1}) we have $ff_{12}=f_1f_2$.
It is easy to see that the solution of this differential equation is the following
$f(x^1,x^2) = A(x^1)B(x^2)$.
Thus $f_1/f=A_1/A$. 
But $f_1/f=\mf{b}'/(2\mf{b})$ and we obtain $2A_1/A=\mf{b}'/\mf{b}$, which in particular gives 
$A^2=\mf{b}$ and without loss of generality
$f(x^1,x^2) = \sqrt{\mf{b}(x^1)}B(x^2)$.
This leads to:
$$f_1=\frac{\mf{b}'}{2\sqrt{\mf{b}}}\,B,\ f_2=\sqrt{\mf{b}}B_2,\ f_{22}= \sqrt{\mf{b}}B_{22},\ f_{11}=\frac{B}{2\sqrt{\mf{b}}}\left(\mf{b}''-\frac{(\mf{b}')^2}
{2\mf{b}}\right).$$
Substituting these equalities into (\ref{t2}) we have
$$\frac{B}{2\sqrt{\mf{b}}}\left(\mf{b}''-\frac{(\mf{b}')^2}{2\mf{b}}\right)-\frac{\mf{a}}{\mf{b}}\sqrt{\mf{b}} B_{22}=\frac{(\mf{ab})'}{2\mf{ab}}
\frac{\mf{b}'B}{2\sqrt{\mf{b}}},$$
$$\mf{b}''-\frac{(\mf{b}')^2}{2\mf{b}}=2\mf{a}\frac{B_{22}}{B}+\frac{(\mf{ab})'}{2\mf{ab}}\,\mf{b}'.$$
The last equality implies
\begin{equation}
\label{BC}
\frac{B_{22}}{B}=C=const.
\end{equation}
Therefore we have
\begin{equation}
\label{rb}
\mf{bb}''-(\mf{b}')^2=2\mf{ab}C+\frac{\mf{a}'\mf{b}'\mf{b}}{2\mf{a}}\,.
\end{equation}
Rewriting this equation in the form
$$-\mf{a}'+2\mf{a}\left(\frac{\mf{b}''}{\mf{b}'}-\frac{\mf{b}'}{\mf{b}}\right)=\frac{4C\mf{a}^2}{\mf{b}'}$$
we have Bernoulli's equation with respect to the unknown function $\mf{a}$ which leads to 
\begin{equation}
\label{rab}
\mf{a}=\frac{(\mf{b}')^2}{\mf{b}(D\mf{b}-4C)}\,,\ D\in\mathbb R.
\end{equation}
Comparing this equality with Lemma 3.1 we have the following.
\begin{cor}
Let $(M,g)$ be a $2$-dimensional manifold with the metric
$$g_{11}=\mf{a}(x^1),\ g_{22}=\mf{b}(x^1),\ g_{12}=0$$
and let functions $\mf{a}$ and $\mf{b}$ satisfy (\ref{rb}). Then the function $\mf{a}$ satisfies the relation
(\ref{rab}) and the Gauss curvature of $M$ is constant, namely
\begin{equation}
\label{kD}
\kappa_G=-\frac D4\,.
\end{equation}
\end{cor}

Taking into account the equality $f(x^1,x^2)=\sqrt{\mf{b}(x^1)}B(x^2)$ and (\ref{Ff}) we have
\begin{equation}
\label{FbB}
F(x^1,x^2)=\mf{b}(x^1)B^2(x^2)\,.
\end{equation}
Computing once more $T_{22}$ and using (\ref{FbB}),
(\ref{BC}) and (\ref{rab}) we obtain $T_{22}=(D/2)\mf{b}^2B^2$. Thus, in view of (\ref{Tg})(ii) we have
${\rm tr}(T)=\widehat{g}^{11}T_{11}+\widehat{g}^{22}T_{22}=(2/\mf{b}) T_{22}=D\mf{b}B^2$, i.e.
\begin{equation}
\label{trT}
{\rm tr}(T)=DF\,.
\end{equation}
For a function $B$ satisfying (\ref{BC}) we have the following.
\begin{rem}
Let a function $B:\mathbb R\to\mathbb R$ satisfies
$B''(t) / B(t) = C = const$. Then:\\
(i)\hspace{3mm}$(B')^2-CB^2=const$.,\\
(ii)\hspace{2mm}the function $B$, according to $C$, is of the form:\\
(a)\hspace{3mm}$C>0$, then $B(t)=C_1e^{\sqrt{C}t}+C_2e^{-\sqrt{C}t}\,,\ C_1,C_2\in\mathbb R$,\\
(b)\hspace{3mm}$C<0$, then $B(t)=C_1\cos{\sqrt{-C}t}+C_2\sin{\sqrt{-C}t}\,,\ C_1,C_2\in\mathbb R$,\\
(c)\hspace{3mm}$C=0$, then $B(t)=C_1t+C_2\,,\ C_1,C_2\in\mathbb R$.
\end{rem}

Concerning the conformal flatness of warped product manifolds we have the following.
\begin{rem}
Let $\widehat{M} \times _F \widetilde{N}$ be the warped product manifold with a $2$-dimensional manifold $(\widehat{M},\widehat{g})$
and an $(n-2)$-dimensional fiber $(\widetilde{N},\widetilde{g})$, $n\geqslant 4$, and a warping function $F$, and let
$(\widetilde{N},\widetilde{g})$ be a semi-Riemannian space of constant curvature, when $n\geqslant 5$.
The local components $C_{hijk}$ of the Weyl conformal curvature tensor C of $\widehat{M} \times _F \widetilde{N}$ are expressed by
(see \cite{DGJZ})
\begin{eqnarray*}
C_{abcd}&=&\frac{(n-3)\rho_0}{n-1}G_{abcd}\,,\ \ C_{\alpha bc\delta}=-\frac{(n-3)\rho_0}{(n-2)(n-1)}G_{\alpha bc\delta}\,,\\
C_{\alpha\beta\gamma\delta}&=&\frac{2\rho_0}{(n-2)(n-1)}G_{\alpha\beta\gamma\delta}\,,\ \ C_{abc\delta}=C_{ab\gamma\delta}=
C_{a\beta\gamma\delta}=0\,,
\end{eqnarray*}
\begin{equation}
\label{ro0}
\rho_0=\frac{\widehat{\kappa}}{2}+\frac{\widetilde{\kappa}}{(n-3)(n-2)F}+\frac{tr(T)}{2F}-\frac{\Delta_1F}{4F^2}\,.
\end{equation}
\end{rem}
Hence, $\widehat{M} \times _F \widetilde{N}$ is conformally flat if and only if $\rho_0=0$.\\
Taking into account (\ref{rab}) we find
\begin{equation}
\label{Delta1}
\Delta_1 F=\frac{1}{\mf{a}}F^2_1+\frac{1}{\mf{b}}F^2_2=\mf{b}B^2\left(B^2(D\mf{b}-4C)+4B^2_2\right)\,.
\end{equation}
Thus, by using (\ref{ro0}), (\ref{kD}) and ({\ref{trT}), we obtain
\begin{eqnarray*} 
\rho_0&=&-\frac D4+\frac{\widetilde{\kappa}}{(n-3)(n-2)}\,\frac 1F+\frac{DF}{2F}-\frac 1F\frac 14(B^2(D\mf{b}-4C)+4B^2_2)=\\
&=&\frac D4+\frac{\widetilde{\kappa}}{(n-3)(n-2)F}-\frac 1F\left(B^2(\frac D4 \mf{b}-C)+B^2_2\right)=\\
&=&\frac 1F\left(\frac{\widetilde{\kappa}}{(n-3)(n-2)}-B^2_2+CB^2\right)\,.
\end{eqnarray*}
Therefore, the equality $\rho_0=0$ is equivalent to
\begin{equation}
\label{Cflat}
\frac{\widetilde{\kappa}}{(n-3)(n-2)}=B^2_2-CB^2\,.
\end{equation}
Now we consider the following problem: when the manifold $\widehat{M} \times _F \widetilde{N}$ is quasi-Einsteinian or Einsteinian.
In virtue of (\ref{AL2}) we have
$$S_{ab}=\widehat{S}_{ab}-\frac{n-2}{F}\,T_{ab}=\left(\frac{\widehat{\kappa}}{2}-\frac{n-2}{2F}\,\frac{{\rm tr}(T)}{2}\right)g_{ab}$$
and in view of $\widehat{\kappa}_G=\widehat{\kappa}/2$, (\ref{kD}) and (\ref{trT}) we get
\begin{equation}
\label{Sab}
S_{ab}=-\frac{n-1}{4}\,Dg_{ab}\,.
\end{equation}
Similarly,
$$S_{\alpha\beta}=\widetilde{S}_{\alpha\beta}-\frac 12 \left({\rm tr}(T)+\frac{n-3}{2}\,\frac{\Delta_1 F}{F}\right)\widetilde{g}_{\alpha\beta}=
\left(\frac{\widetilde{\kappa}}{n-2}-\frac{{\rm tr}(T)}{2}-\frac{n-3}{4F}\Delta_1F\right)\frac 1F g_{\alpha\beta}\,.$$
Taking into account (\ref{Delta1}) we find
\begin{equation}
\label{Salfa}
S_{\alpha\beta}=\frac{1}{\mf{b}B^2}\left(\frac{\widetilde{\kappa}}{n-2}-\frac{n-1}{4}\,D\mf{b}B^2+(n-3)(B^2C-B^2_2)\right)g_{\alpha\beta}\,.
\end{equation} 
Equalities (\ref{quasi02}), (\ref{Sab}) and (\ref{Salfa}) imply that $\widehat{M} \times _F \widetilde{N}$ cannot be quasi-Ensteinian and will be Einsteinian
if and only if
$$\frac{\widetilde{\kappa}}{n-2}-\frac{n-1}{4}\,D\mf{b}B^2+(n-3)(B^2C-B^2_2)=-\frac{n-1}{4}\,D\mf{b}B^2\,,$$
which reduces to the equality (\ref{Cflat}).\\

\begin{rem}
Let $\widehat{M} \times _F \widetilde{N}$ be the warped product manifold with a $2$-dimensional manifold $(\widehat{M},\widehat{g})$
and an $(n-2)$-dimensional fiber $(\widetilde{N},\widetilde{g})$, $n\geqslant 4$, and a warping function $F$, and let
$(\widetilde{N},\widetilde{g})$ be a semi-Riemannian space of constant curvature, when $n\geqslant 5$.
If $T=\frac{{\rm tr}(T)}{2}\widehat{g}$ on $U=U_S\cap U_C\subset\widehat{M}\times\widetilde{N}$ then we have
(see {\cite[p.12] {DK}})
\begin{eqnarray*}
R_{abcd}&=&\rho_1 G_{abcd},\ \ \ \rho_1=\frac{\widehat{\kappa}}{2},\\
R_{\alpha bc\beta}&=&\rho_2 G_{\alpha bc\beta},\ \ \ \rho_2=-\frac{{\rm tr}(T)}{4F},
\end{eqnarray*}
\begin{eqnarray*}
R_{\alpha\beta\gamma\delta}&=&\rho_3 G_{\alpha\beta\gamma\delta},\ \ \ \rho_3=\frac 1F\left(\frac{\widetilde{\kappa}}{(n-3)(n-2)}-\frac{\Delta_1F}{4F}\right),
\end{eqnarray*}
\begin{eqnarray}
\label{Smi1}
S_{ab}&=&\mu_1g_{ab},\ \ \ \ \ \ \mu_1=\frac{1}{4F}(2F\widehat{\kappa}-(n-2){\rm tr}(T)),\\
\label{Smi2}
S_{\alpha\beta}&=&\mu_2g_{\alpha\beta},\ \ \ \ \ \ \mu_2=\frac 1F\left(\frac{\widetilde{\kappa}}{n-2}-\frac{{\rm tr}(T)}{2}-(n-3)\frac{\Delta_1F}{4F}\right).
\end{eqnarray}
\end{rem}

In the considered case if the equality (\ref{Cflat}) does not hold then $U_S\cap U_C=\widehat{M}\times\widetilde{N}$. Computing now $\mu_1,\mu_2$ and
$\rho_1,\rho_2,\rho_3$ from Remark 4.3, in view of (\ref{trT}) and (\ref{Delta1}), we have
\begin{equation}
\label{mi}
\mu_1=-\frac{n-1}{4}\,D,\ \ \mu_2=\frac 1F\left(\frac{\widetilde{\kappa}}{n-2}-\frac{n-1}{4}\,FD-(n-3)(B^2_2-CB^2)\right)\,.
\end{equation}
Thus
\begin{equation}
\label{mi21}
\mu_2-\mu_1=\frac 1F\left(\frac{\widetilde{\kappa}}{n-2}-(n-3)(B^2_2-CB^2)\right)\,,
\end{equation}
$$\rho_1=\frac{\widehat{\kappa}}{2}=-\frac D4=\rho_2,\ \ \rho_3=\frac{1}{F(n-3)}\left(\frac{\widetilde{\kappa}}{n-2}-(n-3)(B^2_2-CB^2)-(n-3)\frac{FD}{4}\right)\,,$$
$$\rho_3=\frac{1}{n-3} (\mu_2-\mu_1)-\frac D4.$$
According to {\cite[Theorem 4.1] {DK}} $\widehat{M}\times_F\widetilde{N}$ is a Roter manifold, i.e.
(\ref{eq:h7a}) is satisfied, 
with
$\phi = \nu(\rho_1-2\rho_2+\rho_3)$,
$\mu = \nu((\rho_2-\rho_3)\mu_1+(\rho_2-\rho_1)\mu_2)$,
$\eta = \nu(\rho_1\mu_2^2-2\rho_2\mu_1\mu_2+\rho_3\mu_1^2)$,
where $\nu=(\mu_2-\mu_1)^{-2}$.
Thus applying these results we obtain
\begin{equation}
\label{Rot1}
\phi=\frac{1}{(n-3)(\mu_2-\mu_1)}\,,\ \ \mu=-\frac{\mu_1}{(n-3)(\mu_2-\mu_1)}\,,\ \ \eta=\rho_1+\frac{\mu_1^2}{(n-3)(\mu_2-\mu_1)}\,.
\end{equation}
Substituting these equalities to (\ref{LR}) we get
\begin{equation}
\label{L1}
L_R=-\frac D4=\frac{\widehat{\kappa}}{2}=\widehat{\kappa}_G\,.
\end{equation}

We see that $\widehat{M}\times_F\widetilde{N}$ is a Roter type manifold, in particular pseudosymmetric manifold of constant type, and
admits geodesic mapping into $\overline{\widehat{M}}\times_{\overline{F}}\widetilde{N}$, so $\overline{\widehat{M}}\times_{\overline{F}}\widetilde{N}$
is pseudosymmetric manifold of constant type (see Theorem 3.1).
We would like to show that it is also a Roter type manifold. First we compute the components of the tensor $\overline{T}$.\\
We observe, in view of (\ref{gama}), (\ref{Ch2}), (\ref{Chp}) and (\ref{psi}) that
\begin{eqnarray}
\label{gama2}
\overline{\Gamma}^1_{11}&=&\Gamma^1_{11}+2\psi_1=\frac{\mf{a}'}{2\mf{a}}-\frac{q\mf{b}'}{1+q\mf{b}}\,,\nonumber\\
\overline{\Gamma}^2_{12}&=&\Gamma^2_{12}+\psi_1=\frac{\mf{b}'}{2\mf{b}(1+q\mf{b})}\,\\
\nonumber
\overline{\Gamma}^1_{22}&=&\Gamma^1_{22}=-\frac{\mf{b}'}{2\mf{a}}\,.
\end{eqnarray}
Taking into account (\ref{FF}) and (\ref{FbB}) we have 
$$\overline{F}(x^1,x^2)=\frac{p\mf{b}(x^1)B^2(x^2)}{1+q\mf{b}(x^1)}\,,$$
so $\overline{F}_1=\partial_1\overline{F}=\frac{\mf{b}'}{(1+q\mf{b})^2}\,pB^2,\ 
\overline{F}_2=\partial_2\overline{F}=\frac{p\mf{b}}{1+q\mf{b}}\,2BB_2$ and
$$\partial_2\overline{F}_2=\frac{p\mf{b}}{1+q\mf{b}}\,2(B^2_2+BB_{22})=\frac{p\mf{b}}{1+q\mf{b}}\,2(B^2_2+CB^2)\,,$$
in virtue of (\ref{BC}).
Next, using (\ref{gama2}) and (\ref{rab}) we obtain
$$\overline{T}_{12}=\overline{\nabla}_1\overline{F}_2-\frac{1}{2\overline{F}}\overline{F}_1\overline{F}_2=\partial_1\overline{F}_2-\overline{F}_2
\overline{\Gamma}^2_{12}-\frac{1}{2\overline{F}}\overline{F}_1\overline{F}_2=0\,,$$
$$\overline{T}_{22}=\overline{\nabla}_2\overline{F}_2-\frac{1}{2\overline{F}}(\overline{F}_2)^2=\partial_2\overline{F}_2-\overline{F}_1\overline{\Gamma}^1_{22}
-\frac{1}{2\overline{F}}(\overline{F}_2)^2=\frac{D+4qC}{2p}\,\frac{p\mf{b}}{1+q\mf{b}}\,\overline{F}=\frac{D+4qC}{2p}\,\overline{F}\overline{g}_{22}\,,$$
$$\overline{T}_{11}=\overline{\nabla}_1\overline{F}_1-\frac{1}{2\overline{F}}(\overline{F}_{12})^2=\partial_1\overline{F}_1-\overline{F}_1\overline{\Gamma}^1_{11}
-\frac{1}{2\overline{F}}(\overline{F}_1)^2=\frac{pB^2}{(1+q\mf{b})^2}\left(\mf{b}''-\frac{1+2q\mf{b}}{2\mf{b}(1+q\mf{b})}\,(\mf{b}')^2-\frac{\mf{a}'\mf{b}'}{2\mf{a}}\right)\,.$$
But, in virtue of (\ref{rb}) $\ \mf{b}''-\frac{\mf{a}'\mf{b}'}{2\mf{a}}=\frac{(\mf{b}')^2}{\mf{b}}+2\mf{a}C$, and
$$\overline{T}_{11}=\frac{pB^2}{(1+q\mf{b})^2}\left(\frac{(\mf{b}')^2}{2\mf{b}(1+q\mf{b})}+2\mf{a}C\right)=\frac{p(\mf{b}')^2}{\mf{b}(D\mf{b}-4C)(1+q\mf{b})^2}\,
\frac{p\mf{b}B^2}{1+q\mf{b}}\,\frac{D+4qC}{2p}=\overline{g}_{11}\overline{F}\frac{D+4qC}{2p}\,.$$
Thus we see that the following equality holds
\begin{equation}
\label{Tab2}
\overline{T}_{ab}=\frac{D+4qC}{2p}\,\overline{F}\,\overline{g}_{ab}\,.
\end{equation}

From the first equation of (\ref{psiij}), using (\ref{Chp}) and (\ref{Ch2}), we find 
$\psi_{11}=\partial_1\psi_1-\psi_1\Gamma^1_{11}-({\psi}_1)^2$ and in virtue of (\ref{psi}), (\ref{rb}) and
(\ref{rab}), we get
\begin{equation}
\label{psi11}
\psi_{11}=\frac{q(\mf{b}')^2(4C-qD\mf{b}^2-2D\mf{b})}{4\mf{b}(1+q\mf{b})^2(D\mf{b}-4C)}\,.
\end{equation}
Similarly, $\psi_{22}=\partial_2\psi_2-\psi_1\Gamma^1_{12}-(\psi_2)^2=-\psi_1\Gamma^1_{12}$ and
\begin{equation}
\label{psi22}
\psi_{22}=-\frac{q\mf{b}(D\mf{b}-4C)}{4(1+q\mf{b})}\,.
\end{equation}
Using (\ref{Ch2}) we have $\Gamma^1_{\alpha\beta}=-\frac 12\widehat{g}^{11}F_1\widetilde{g}_{\alpha\beta}=-\frac{1}{2\mf{a}}F_1\widetilde{g}_{\alpha\beta}$.\\
Substituting this equality into $\psi_{\alpha\beta}=\partial_\beta\psi_\alpha-\psi_1\Gamma^1_{\alpha\beta}$ we obtain
\begin{equation}
\label{psi33}
\psi_{\alpha\beta}=-\frac{q\mf{b}B^2(D\mf{b}-4C)}{4(1+q\mf{b})}\,\widetilde{g}_{\alpha\beta}\,.
\end{equation}
In the same manner we easily get
\begin{equation}
\label{psi4}
\psi_{12}=0\,,\ \ \psi_{a\alpha}=0\,.
\end{equation}
Starting with (\ref{S}) we calculate the components of $\overline{S}$.\\
Since $\overline{S}_{11}=S_{11}-(n-1)\psi_{11}$, so using (\ref{Sab}), (\ref{psi11}) and (\ref{rab}) we have
$$\overline{S}_{11}=-\frac{n-1}{4p}(D+4qC)\overline{g}_{11}\,.$$
Similarly, using (\ref{Sab}) and (\ref{psi22}) we obtain 
$$\overline{S}_{22}\ =\ -\frac{n-1}{4p}(D+4qC)\overline{g}_{22}.$$
Now the last two equations and $\overline{S}_{12}=0$ yield
\begin{equation}
\label{Sab2}
\overline{S}_{ab}=-\frac{n-1}{4p}(D+4qC)\overline{g}_{ab}\,.
\end{equation}
Taking into account (\ref{Salfa}) and (\ref{psi33}) we get
\begin{equation}
\label{Salfa2}
\overline{S}_{\alpha\beta}=\left(\frac{\widetilde{\kappa}}{n-2}+(n-3)(CB^2-B^2_2)-\frac{n-1}{4}\,\mf{b}B^2\,\frac{D+4qC}{1+q\mf{b}}\right)
\frac{1}{\overline{F}}\,\overline{g}_{\alpha\beta}\,.
\end{equation}
Finally, in view of (\ref{psi4}) we have
\begin{equation}
\label{Saalfa2}
\overline{S}_{a\alpha}=0\,.
\end{equation}
On the other hand (\ref{AL2}) leads to 
$$\overline{S}_{ab}\ =\ \overline{\widehat{S}}_{ab}-\frac{n-2}{2\overline{F}}\overline{T}_{ab}.$$
Thus substituting into this equality (\ref{Sab2}) and (\ref{Tab2}) we obtain
$$\overline{\widehat{S}}_{ab}=-\frac{D+4qC}{4p}\,\overline{g}_{ab}\,,$$
which implies
\begin{equation}
\label{kh2}
\overline{\widehat{\kappa}}=-\frac{D+4qC}{2p}\,.
\end{equation}
Equalities (\ref{Sab2}), (\ref{Salfa2}) and (\ref{Saalfa2}) imply that $\overline{\widehat{M}} \times _{\overline{F}} \widetilde{N}$ cannot be quasi-Einsteinian
and will be Einsteinian if and only if
$$-\frac{n-1}{4p}(D+4qC)\,\frac{p\mf{b}B^2}{1+q\mf{b}}=\frac{\widetilde{\kappa}}{n-2}+(n-3)(CB^2-B^2_2)-\frac{n-1}{4p}\,\mf{b}B^2\frac{D+4qC}{1+q\mf{b}},$$
which reduces to the equality (\ref{Cflat}).\\
According to Remark 4.2 $\overline{\widehat{M}} \times _{\overline{F}} \widetilde{N}$ is conformally flat if and only if $\overline{\rho}_0=0$, i.e.
\begin{equation}
\label{ro2}
\frac{\overline{\widehat{\kappa}}}{2}\,\overline{F}+\frac{\widetilde{\kappa}}{(n-3)(n-2)}+\frac{{\rm tr}(\overline{T})}{2}
=\frac{\overline{\Delta}_1\overline{F}}
{4\overline{F}}\,.
\end{equation}
Substituting (\ref{kh2}) into (\ref{Tab2}) we have $\overline{T}=-\overline{\widehat{\kappa}}\,\overline{F}\overline{g}$ which gives
${\rm tr}(\overline{T})=-2\overline{\widehat{\kappa}}\,\overline{F}$ and
$$\frac{\overline{\widehat{\kappa}}}{2}\overline{F}+\frac{{\rm tr}(\overline{T})}{2}=-\frac{\overline{\widehat{\kappa}}}{2}\overline{F}=
\frac{\mf{b}B^2(D+4qC)}{4(1+q\mf{b})}\,.$$
Next, since $\overline{\Delta}_1\overline{F}=\frac{1}{\overline{g}_{11}}(\overline{F}_1)^2+\frac{1}{\overline{g}_{22}}(\overline{F}_2)^2$, so using (\ref{rab}) we
easily derive that
$$\frac{\overline{\Delta}_1\overline{F}}{4\overline{F}}=B^2_2+\frac{D\mf{b}-4C}{1+q\mf{b}}\,B^2\,.$$
Thus the equality (\ref{ro2}) takes the form
$$\frac{\widetilde{\kappa}}{(n-3)(n-2)}=-\frac{\mf{b}B^2(D+4qC)}{4(1+q\mf{b})}+B^2_2+\frac{(D\mf{b}-4C)B^2}{1+q\mf{b}}=B^2_2-CB^2\,,$$
i.e. the equality (\ref{Cflat}).
Therefore, if the equality (\ref{Cflat}) does not hold then $U_{\overline{S}}\cap U_{\overline{C}}=\overline{\widehat{M}}\times\widetilde{N}$.
Applying Remark 4.3 to the warped product $\overline{\widehat{M}} \times _{\overline{F}} \widetilde{N}$ and using earlier results we have
$$\overline{\rho}_1=\frac{\overline{\widehat{\kappa}}}{2}=\overline{\rho}_2\,,$$
$$\overline{\mu}_1=\frac{n-1}{2}\,\overline{\widehat{\kappa}}\,,\ \ \overline{\mu}_2=\frac{1}{\overline{F}}\,\frac{\widetilde{\kappa}}{n-2}+
\frac{1}{\overline{F}}(n-3)(CB^2-B^2_2)+\frac{n-1}{2}\,\overline{\widehat{\kappa}}\,,$$
so,
$$\overline{\mu}_2-\overline{\mu}_1=\frac{1}{\overline{F}}\left(\frac{\widetilde{\kappa}}{n-2}+(n-3)(CB^2-B^2_2)\right)$$
$$\overline{\rho}_3=\frac{1}{\overline{F}}\left(\frac{\widetilde{\kappa}}{(n-2)(n-3)}+\overline{F}\frac{\overline{\widehat{\kappa}}}{2}+CB^2-B^2_2\right)
=\frac{1}{n-3}(\overline{\mu}_2-\overline{\mu}_1)+\frac{\overline{\widehat{\kappa}}}{2}\,.$$
According to {\cite[Theorem 4.1] {DK}} $\overline{\widehat{M}} \times _{\overline{F}} \widetilde{N}$ is a Roter manifold, i.e.
$$\overline{R}=\frac{\overline{\phi}}{2}\, \overline{S}\wedge \overline{S} + \overline{\mu}\, \overline{g}\wedge \overline{S} +
\frac{\overline{\eta}}{2}\,\overline{g}\wedge \overline{g}\,,$$
with
$\overline{\phi} = \overline{\nu}(\overline{\rho}_1-2\overline{\rho}_2+\overline{\rho}_3)$,
$\overline{\mu} = \overline{\nu}((\overline{\rho}_2-\overline{\rho}_3)\overline{\mu}_1+(\overline{\rho}_2-\overline{\rho}_1)\overline{\mu}_2)$,
$\overline{\eta} = \overline{\nu}(\overline{\rho}_1\overline{\mu}_2^2-2\overline{\rho}_2\overline{\mu}_1\overline{\mu}_2+\overline{\rho}_3\overline{\mu}_1^2)$,
where $\overline{\nu}=(\overline{\mu}_2-\overline{\mu}_1)^{-2}$. 
Applying calculated expressions for $\overline{\rho}_1, \overline{\rho}_2, \overline{\rho}_3$ we obtain
$$\overline{\phi}=\frac{1}{(n-3)(\overline{\mu}_2-\overline{\mu}_1)}\,,\ \ \overline{\mu}=-\frac{\overline{\mu}_1}{(n-3)(\overline{\mu}_2
-\overline{\mu}_1)}\,,\ \ \overline{\eta}=\overline{\rho}_1+\frac{\overline{\mu}_1^2}{(n-3)(\overline{\mu}_2-\overline{\mu}_1)}\,.$$
Substituting these equalities into (\ref{LR}) we get
\begin{equation}
\label{L2}
L_{\overline{R}}=-\frac{D+4qC}{4p}=\frac{\overline{\widehat{\kappa}}}{2}=\overline{\widehat{\kappa}}_G\,.
\end{equation}
Thus we have proved the following.
\begin{theo}
Let $(\widehat{M},\widehat{g})$ be a 2-dimensional manifold with a metric $\widehat{g}$ given by
$$\widehat{g}_{11}=\mf{a}(x^1)\,,\ \ \widehat{g}_{22}=\mf{b}(x^1)\,,\ \ \widehat{g}_{12}=0\,,$$
where
$$\mf{a}=\frac{(\mf{b}')^2}{\mf{b}(D\mf{b}-4C)}\,,\ C,D\in\mathbb R\,.$$
Next let $(\widetilde{N},\widetilde{g})$ be an $(n-2)$-dimensional, $n\geqslant 4$, semi-Riemannian space of constant
curvature, when $n\geqslant 5$ and let $\widehat{M}\times_F\widetilde{N}$ be the warped product manifold with warping function
$$F=F(x^1,x^2)=\mf{b}(x^1)B^2(x^2)\,,$$
where $B$ is a function described in Remark 4.1 such that the equality (\ref{Cflat}) does not hold.
Then $\widehat{M}\times_F\widetilde{N}$ is a Roter type manifold which admits a non-trivial
geodesic mapping onto a warped product manifold $\overline{\widehat{M}}\times_{\overline{F}}\widetilde{N}$, where 
$(\overline{\widehat{M}},\overline{\widehat{g}})$ is a manifold geodesically related to $(\widehat{M},\widehat{g})$ with a metric 
$\overline{\widehat{g}}$ given by
$$\overline{\widehat{g}}_{11}=\frac{p\,\mf{a}}{(1+q\mf{b})^2}\,,\ \ \overline{\widehat{g}}_{22}=\frac{p\mf{b}}{1+q\mf{b}}\,,\ \ \overline{\widehat{g}}_{12}=0\,,\ 
p,q\in\mathbb R\,,$$
and warping function $\overline{F} = \frac{p}{1+q\mf{b}} F$.
Moreover, $\overline{\widehat{M}}\times_{\overline{F}}\widetilde{N}$ is also a Roter type manifold.
\end{theo}  
\begin{prop}
Under above assumptions we have
$$L_R=-\frac D4\,,\ \ \ L_{\overline{R}}= -\frac{D+4qC}{4p}\,,$$
so both manifolds $(M,g)$ and $(\overline{M},\overline{g})$ are pseudosymmetric of constant type. Moreover we have
\begin{equation}
\label{Lkappa}
L_R-\frac{\kappa}{n(n-1)}=\frac{p}{1+q\mf{b}}\left(L_{\overline{R}}-\frac{\overline{\kappa}}{n(n-1)}\right)\,.
\end{equation}\end{prop}
\begin{proof}
The equalities (\ref{L1}) and (\ref{L2}) give the first part of the assertion.
Using (\ref{Smi1}), (\ref{Smi2}) and (\ref{mi}) we obtain
\begin{eqnarray*}
\kappa&=&g^{ab}S_{ab}+g^{\alpha\beta}S_{\alpha\beta}=g^{ab}g_{ab}\mu_1+g^{\alpha\beta}g_{\alpha\beta}\mu_2=2\mu_1+(n-2)\mu_2=\\
&=&-\frac{2(n-1)D}{4}+\frac{n-2}{F}\left(\frac{\widetilde{\kappa}}{n-2}-\frac{n-1}{4}\,DF-(n-3)(B^2_2-CB^2)\right)=\\
&=&-\frac{n(n-1)}{4}\,D+\frac{\widetilde{\kappa}}{F}-\frac{(n-3)(n-2)}{F}(B^2_2-CB^2)
\end{eqnarray*}
and in virtue of (\ref{L1}) we have
$$\kappa=n(n-1)L_R+\frac{\widetilde{\kappa}}{F}-\frac{(n-3)(n-2)}{F}(B^2_2-CB^2)\,.$$
Similarly, using analogous equations for $(\overline{M},\overline{g})$ we get
$$\overline{\kappa}=n(n-1)L_{\overline{R}}+\frac{\widetilde{\kappa}}{\overline{F}}-\frac{(n-3)(n-2)}{\overline{F}}(B^2_2-CB^2)\,.$$
Comparing two last relations we have
$$L_R-\frac{\kappa}{n(n-1)}=\frac{\overline{F}}{F}\left(L_{\overline{R}}-\frac{\overline{\kappa}}{n(n-1)}\right)\,,$$
and taking into account (\ref{FF}) we obtaion (\ref{Lkappa}).
\end{proof}
Roter type manifolds satisfy various curvature conditions of pseudosymmetry type. We find formulas for the functions $L_C$ and $L$ in 
(\ref{LCC}) and (\ref{LSR}).\\
Taking into account (\ref{S2}) and (\ref{LCC}) we get
$$L_C=L_R-\frac{\kappa}{n-1}+\frac{1-(n-2)\mu}{(n-2)\phi}\,.$$
Using now (\ref{Rot1}), (\ref{mi21}) and (\ref{mi}), after standard calculation we obtain
\begin{equation}
\label{LCR}
\frac{(n-2)^2}{n}\,L_C=L_R-\frac{\kappa}{n(n-1)}\,.
\end{equation}
Similarly (for $\overline{g}$) we have
$$\frac{(n-2)^2}{n}\,L_{\overline{C}}=L_{\overline{R}}-\frac{\overline{\kappa}}{n(n-1)}\,.$$
Thus, in virtue of (\ref{Lkappa}) we get
$$L_C=\frac{\overline{F}}{F}\,L_{\overline{C}}=\frac{p}{1+q\mf{b}}\,L_{\overline{C}}\,.$$
Taking into account (\ref{LSR}), (\ref{Rot1}) and (\ref{mi}) we easily obtain
$L=-(n-2)L_R$ and, similarly $\overline{L}=-(n-2)L_{\overline{R}}$.
Thus we have the following.
\begin{cor}
For manifolds $(M,g)$ and $(\overline{M},\overline{g})$ satisfying assumptions of the Theorem 4.1 we have\\
(i) $C\cdot C=L_C\,Q(g,C)\,,\ \ \overline{C}\cdot\overline{C}=L_{\overline{C}}\,Q(\overline{g},\overline{C})$,\\
where $L_C$ is given by (\ref{LCR}) and $L_C = (p /(1+q\mf{b}))\, L_{\overline{C}}$. \\
(ii) $R\cdot R=Q(S,R)-(n-2)L_R\,Q(g,C),\ \ \ \overline{R}\cdot\overline{R}=Q(\overline{S},\overline{R})-(n-2)L_{\overline{R}}\,Q(\overline{g},\overline{C})\,.$
\end{cor}

As we mentioned at the end of Section 1, we continue investigations of geodesic mappings in Roter spaces and, for example, we obtained
\begin{rem} \cite{DH-2019}
Let $(M,g)$ be a pseudosymmetric non-semi-symmetric semi-Riemannian manifold admitting a non-trivial geodesic mapping
onto a Roter space $(\overline{M},\overline{g})$. Then we have the following
$$(\overline{\kappa}\overline{\phi} + n\overline{\mu})B -({\rm tr}(B)\overline{\phi}+{\rm tr}(\psi)\overline{\mu})\overline{S}+(\overline{\kappa}\,\overline{\mu}-
n(L_{\overline{R}}-\overline{\eta}))\psi+({\rm tr}(\psi)(L_{\overline{R}}-\overline{\eta})-{\rm tr}(B)\overline{\mu})\overline{g}=0\,,$$
where $\psi$ and $B$ are (0,2)-tensors with components given by (\ref{psiij}) and $B_{mk}=\psi_{mr}\overline{S}^r_k$, respectively.
\end{rem}
Moreover, we found the sufficient conditions for the manifold $(M,g)$ to be also a Roter space.

\vspace{5mm}

{\bf{Acknowledgements.}} The first named author is supported
by a grant of the Wroc\l aw University of Environmental and Life Sciences (Poland).

\vspace{5mm}

\noindent
Ryszard Deszcz
\newline
Department of Mathematics
\newline 
Wroc\l aw Univeristy of Environmental and Life Sciences
\newline
Grunwaldzka 53, 50-357 Wroc\l aw, Poland
\newline
{\sf{E-mail: Ryszard.Deszcz@upwr.edu.pl}}

\vspace{2mm}

\noindent
Marian Hotlo\'{s}
\newline
Department of Applied Mathematics
\newline 
Wroc\l aw Univeristy of Science and Technology
\newline
Wybrze\.{z}e Wyspia\'{n}skiego 27, 50-370 Wroc\l aw, Poland 
\newline
{\sf{E-mail: Marian.Hotlos@pwr.edu.pl}}

\end{document}